\documentclass[a4paper, 12pt,pdftex
]{amsart}

\usepackage[all]{xy}
\usepackage[latin1]{inputenc}
\usepackage{amsmath,amssymb,amsfonts, amsthm, amscd, latexsym, marginnote}

\usepackage{rotating}
\usepackage[english]{babel}

\usepackage{graphicx}
\CompileMatrices

\usepackage{tikz}
\usepackage[hyperindex,breaklinks,pdftex]{hyperref} 

\newtheorem{thm}{Theorem}[section]

\newtheorem{cor}[thm]{Corollary}
\newtheorem{lem}[thm]{Lemma}
\newtheorem{prop}[thm]{Proposition}
\theoremstyle{definition}
\newtheorem{defin}[thm]{Definition}
\theoremstyle{definition}
\newtheorem{exm}[thm]{Example}
\newtheorem{rmk}[thm]{Remark}

\newenvironment{mucca}[0]{}{}

\newcommand{\q}[1]{\mathbf{#1}}

\newcommand\Q{\mathbb Q}

\newcommand\Id{\mathrm{Id}}
\newcommand\GL{\mathrm{GL}}

\newcommand{\A}{{\mathbb A}}
\newcommand{\B}{{\mathbb B}}
\newcommand{\R}{{\mathbb R}}
\newcommand{\C}{{\mathbb C}}

\newcommand{\Z}{{\mathbb Z}}

\newcommand{\Ab}[0]{\mathrm{Ab}}

\newcommand{\mH}{\mathcal{H}}
\newcommand{\mC}{\mathcal{C}}
\newcommand{\mL}{\mathcal{L}}
\newcommand{\mF}{\mathcal{F}}
\newcommand{\mS}{\mathcal{S}}
\newcommand{\Br}{\mathcal{B}}
\newcommand{\Sal}{\mathrm{Sal}}
\newcommand{\Aut}{\mathrm{Aut}}
\newcommand{\F}[0]{\mathbb{F}}

\newcommand{\ph}[0]{\varphi}

\newcommand{\into}[0]{\hookrightarrow}

\newcommand{\bin}[2]{  \left( \!\!\! \begin{array}{c} #1 \\ #2  \end{array} \!\!\!
  \right)  }

\newcommand{\qbin}[2]{  \left[ \!\!  \begin{array}{c}  #1 \\ #2
    \end{array} \!\!  \right]  }
    
\def\qed{\ifmmode $\Box$ \else{\unskip\nobreak\hfil
\penalty50\hskip1em\null\nobreak\hfil $\Box$
\parfillskip=0pt\finalhyphendemerits=0\endgraf}\fi}

\newcommand{\eq}[1][r]
       {\ar@<-3pt>@{->}[#1]
        \ar@<-1pt>@{}[#1]|<{}="gauche"
        \ar@<+0pt>@{}[#1]|-{}="milieu"
        \ar@<+1pt>@{}[#1]|>{}="droite"
        \ar@/^2pt/@{-}"gauche";"milieu"
        \ar@/_2pt/@{-}"milieu";"droite"}
\newcommand{\imm}[1][r] {\ar@{^{(}->}[#1]}

\newcommand{\pmu}[0]{{\pm 1}}

\begin{document}

\title[Salvetti complex, spectral sequences, Artin groups]
{Salvetti complex, spectral sequences \\and cohomology of Artin groups}

\author[F.~Callegaro]{Filippo Callegaro}
\address{Dipartimento di Matematica \\ Universit\`a di Pisa \\ Italy}
\email{callegaro@dm.unipi.it}

\date{\today}
\begin{abstract}
%
%
%
%

\textbf{English version:} The aim of this short survey is to give a quick 
introduction to the Salvetti complex as a tool for the study of the cohomology
of Artin groups. In particular we show how a spectral sequence 
induced by a filtration on the complex 
provides a very natural and useful method to study recursively the cohomology 
of Artin groups, simplifying many 
computations. In the last section some examples of applications are presented.

\textbf{French version:}
Le but de ce travail est de donner une br\`eve introduction aux
complexes de Salvetti
comme instrument pour \'etudier la cohomologie des groupes d'Artin.
Nous montrons comment une suite spectrale donn\'ee par une filtration
sur le complexe
va d\'efinir une m\'ethode, utile ainsi que tr\`es naturelle, pour \'etudier
r\'ecursivement la cohomologie
des groupes d'Artin, avec une grande simplification dans les calculs.
Dans la derni\`ere partie du travail nous pr\'esentons des exemples d'applications.
\end{abstract}
\maketitle

\section{Introduction} \label{s:intro}

The classical braid group has been defined in 1925 by Artin (\cite{Artin25}). 
In 1962 Fox and Neuwirth \cite{fox} proved that the group defined by Artin is the fundamental
group of the configuration space $C(\R^2, n)$ of unordered $n$-tuples of distinct points in the 
real plane.  
A more general algebraic definition of Artin groups can be given starting from
the standard presentation of a Coxeter group $W.$ 

Given a Coxeter group $W$ acting on a real vector space $V$ we can consider 
the collection $\mH_W$ of all the hyperplanes $H$ 
which are fixed by 
a reflection $\rho \in W.$ This collection is the \emph{reflection arrangement} of $W.$
In \cite{briesk2} Brieskorn proved that 
the fundamental group of the regular orbit space with respect to the
action of the group $W$  
on the complement of a complexified reflection arrangement is the Artin group $A$ associated to $W.$

We illustrate the case of the braid group, 
that can be considered as the leading example of this construction.
We will use it for several other examples along this paper.
We consider the action, by permuting coordinates, 
of the symmetric group on $n$ letters $\mathfrak{S}_n$ on 
the complex vector space $\C^n$. 
If we restrict this action of $\mathfrak{S}_n$ 
to the space of ordered $n$-tuples of distinct points $F(\C,n)$
we obtain a free and properly discontinuous action. 
The space $F(\C,n)$ is the complement
of the union of the hyperplanes of the form $H_{ij} = \{z_i =z_j \}$ in $\C^n.$
The quotient $C(\C,n)=F(\C,n)/\mathfrak{S}_n$ is the regular orbit space for $\mathfrak{S}_n$ 
and hence its fundamental group is the braid group on $n$ strands $\Br_n,$ 
that is the Artin group associated to $\mathfrak{S}_n.$  

The result of Brieskorn mentioned above
shows the important relation between Artin groups and arrangements of hyperplanes, 
since an Artin group is the fundamental group of a quotient of the complement of a reflection
arrangement. 

Research on arrangements of hyperplanes started with the works 
of E. Fadell, R. Fox, L. Neuwirth, V.I. Arnol$'$d, 
E. Brieskorn, T. Zaslavsky, K. Saito, P. Deligne, A. Hattori and later P. Orlik, L. Solomon, 
H. Terao, M. Goresky, R. MacPherson, C. De Concini, C. Procesi, 
M. Salvetti, R. Stanley, R. Randell, G. Lehrer, A. Bj\"orner, G. Ziegler and many others. 
A basic reference for the subject is \cite{ot}. A more recent reference with many recent developments 
and a wide bibliography on the theory of hyperplane arrangements 
is given by the book (still work in progress) \cite{cxarr}.
Given an arrangement $\mH,$ an important combinatorial invariant is the \emph{intersection lattice}
$L(\mH),$ that is the poset of non-empty intersections of elements of $\mH$ ordered 
by reverse inclusion. 
One of the main problems in the study of arrangements
is to understand the relation between the topology of the complement of the arrangement
and 
its intersection lattice.
For a real arrangement we have a finer combinatorial invariant, the \emph{face poset}
(see Definition \ref{d:face_poset} and \cite{ot}).
In \cite{salv87} Salvetti
introduced a CW-complex $\Sal(\mH)$ associated to a real arrangement $\mH$ and 
determined by the face poset of $\mH.$ He proved that this complex is 
homotopy equivalent to the complement of the complexified arrangement.
Moreover if $\mH$ is associated to a reflection group $W,$ the
group $W$ acts on the complex $\Sal(\mH)$ and the quotient complex $X_W$ is homotopy
equivalent to the regular orbit space of $W$ (see \cite{salvetti94, decsal96}). An extension of these results
for an oriented matroid can be found in \cite{gelryb}. For a general complex
arrangement, in \cite{bz} Bj\"orner and Ziegler construct a finite regular cell
complex with the homotopy type of the complement of the arrangement.

In this short survey we present some methods and useful tools for the study of Artin groups 
through the Salvetti complex. A natural filtration of the complex allows to define a
spectral sequence that can be very helpful in several homology and cohomology computations.
In particular we can use the Salvetti complex to compute the cohomology of Artin groups, 
either with constant coefficients or with a local system of coefficients.
The computation of the cohomology of the Milnor fiber, which is related to a very
interesting abelian local system over a Laurent polynomial ring, plays a special role in this context. 

In Section \ref{s:salvetti} we recall our main notation for arrangement of hyperplanes and the Salvetti complex.
We try to keep the notation introduced in \cite{paris}. 
In Section \ref{s:filtration} we give a general introduction to computations using 
a spectral sequence that arises from
a natural filtration of the Salvetti complex. Finally
in Section \ref{s:examples} we provide a few examples that show how the computations via this spectral sequence
can be applied to the study of the cohomology and homology of braid groups, 
providing a simpler or shorter proof 
for previously known results. A first example is given 
in Section \ref{ss:classical} where we provide a shorter
proof of Fuks's result (see \cite{fuks}) on the homology of braid groups mod $2.$ 
Another example is in Section \ref{ss:rat}: we compute the rational 
cohomology of the commutator subgroup of the braid group
giving a new proof of some results already appeared in \cite{fren}, \cite{mar} and \cite{dps}.
In Section \ref{ss:affine} we show how the Salvetti complex can be modified in order to study recursively affine 
type Artin groups. In Section \ref{ss:nonab} we show how it can be used for computer 
investigations providing the example of a non-abelian local system.

\subsection*{Acknowledgment}

The author would like to thank
the organizing and scientific committees of the 
School ``Arrangements in Pyr\'en\'ees'' held in June 2012 in Pau, 
where the idea of these notes started.
\section{Hyperplane arrangements, Artin groups and Salvetti complex} \label{s:salvetti}

\subsection{Hyperplane arrangements}
We recall some definitions and results on hyperplane arrangements and Artin groups. 
We follow the notation 
of \cite{paris}
and we refer to it for a more detailed introduction. We refer to \cite{ot} for a general
introduction on the subject of hyperplane arrangements.

Let $I$ be an open convex cone in a finite dimensional real vector space $V.$ 
\begin{defin}
A real \emph{hyperplane arrangement} in $I$ is a family $\mH$ of real affine hyperplanes of 
$V$ such that each hyperplane of $\mH$ intersects $I$ and the family $\mH$ is locally finite in $I.$
\end{defin}
\begin{defin} \label{d:face_poset}
A real hyperplane arrangement $\mH$ induces a stratification on the convex cone $I$ into \emph{facets}. 
Given two points $x$ and $y$ in $I$ we say that they belong to the same facet $F$ if for every 
hyperplane $H \in \mH$ either $x \in H$ and $y \in H$ or $x$ and $y$ belong to the same connected component
of $I \setminus H.$
We call the set of all facets $\mS$  the \emph{face poset} of $\mH$ and
we equip $\mS$ with the partial order given by  $F > F'$ if and only if 
$\overline{F} \supset F'.$
\end{defin}

A \emph{face} is a codimension $1$ facet, i.~e. a facet that is contained in 
exactly one hyperplane of the arrangement. A \emph{chamber} of the arrangement is a maximal facet, that
is a connected component $C$ of the complement $$I \setminus \cup_{H \in \mH} H.$$

Let $H$ be a real affine hyperplane and let $v(H)$ be its underlying 
vector space: the \emph{complexified hyperplane}
$H_\C$ is the complex affine hyperplane $H_{\C} := \{ z = x + \imath y, x \in H, y \in v(H) \}$ 
in the complex vector space $V_\C := V \otimes_\R \C.$ 

We recall the definition of the complement of the complexified arrangement: 
$$M(\mH) := (I \oplus \imath V) \setminus \bigcup_{H \in \mH} H_{\C}.$$


Now we consider the case of a Coxeter arrangement. 
Let the couple $(W,S)$ be a Coxeter system and assume that 
the set of generators $S$ is given by linear reflections 
in the vector space $V.$ Then $W$ is a finite subgroup 
of $\GL(V).$ We define the \emph{reflection arrangement} of $W$
as the collection $\mH =\mH_W:= \{H \subset V \mid H 
\mbox{ is the fixed hyperplane of a reflection } \rho \in W \}.$ Given any 
chamber $C$ of the arrangement $\mH$ we 
define the convex cone $I$ associated to $(W,S)$ as the interior of the union
$$
\overline{I} := \bigcup_{w \in W} w\overline{C}.
$$
The complement of the reflection arrangement is given by $M(W):= M(\mH_W).$
The group $W$ acts freely and properly discontinuously on $M(W)$ and we denote by $N(W)$  
the quotient $M(W)/W.$

Let $W$ be a Coxeter group with Coxeter graph $\Gamma.$ 
The fundamental group of the complement $N(W)$ is $A_\Gamma,$ that is the Artin group of type $\Gamma.$ 
The fundamental group of the complement $M(W)$ is the pure Artin group $PA_\Gamma$ (see \cite{bri73}).

\begin{exm} \label{ex:b3}
We consider the example of the group $W=\mathfrak{S}_3$ acting on $I=\R^3$ by permuting coordinates. 
The corresponding reflection arrangement is the given by the hyperplanes $H_{1,2}, H_{1,3}, H_{2,3}$, 
where we define $H_{i,j}= 
\{x \in \R^3 \mid x_i = x_j \}.$
We fix the fundamental chamber $C_0 = \{ x \in \R^3 \mid x_1 < x_2 < x_3 \}$ in the complement of $\mH_W.$
The complement $M(W)$ is the ordered configuration space $F(\C,3),$ while the space $N(W)$ is the unordered 
configuration space $C(\C,3).$ The following is Coxeter graph of $\mathfrak{S}_3$ 
\begin{mucca}
\entrymodifiers={=<4pt>[o][F-]}
\begin{center}
\begin{tabular}{l}
\xymatrix @R=2pc @C=2pc {
 \ar @{-}[r]_(-0.20){s_1} \ar @{-}[r]_(0.80){s_2} & 
}
\end{tabular}
\end{center}\end{mucca}
that is the Coxeter graph of type $\A_2.$ The standard generators of the Coxeter group $\mathfrak{S}_3$ 
are the elements $s_1, s_2$ with relations 
$s_1^2 = s_2^2 = e$ and $s_1s_2s_1 = s_2 s_1 s_2.$ We can 
identify the generator $s_1$ (resp. $s_2$) with the transposition 
$(1,2) \in \mathfrak{S}_3 \mbox{(resp. }(2,3)\mbox{)}.$ 
The fundamental group of $C(\C,3)$ is the classical braid group on three strands $\Br_3$ 
and the fundamental group of $F(\C,3)$ is the pure braid group braid group on three strands $\mathcal{P}\Br_3.$
The braid group $\Br_3$ is generated by the elements $\sigma_1, \sigma_2$ with relation 
$\sigma_1\sigma_2\sigma_1 = \sigma_2\sigma_1\sigma_2$ (see, for example, \cite{bri73}).

\end{exm}

\subsection{The Salvetti complex} 
The key geometric object that we consider in this survey 
is the Salvetti complex. This is a CW-complex which has the 
homotopy type of the complement $M(\mH).$ Moreover in the case of 
finite arrangements the Salvetti complex has a finite number of cells. 
Its explicit description and the simple structure, especially in the
case of reflection arrangements, turn out to be very important for 
filtrations and recursive arguments.

In this survey we don't provide an explicit definition of the Salvetti complex. 
The reader interested on the subject can find the original definition in \cite{salv87}.
An extended definition of can be found in \cite{paris}. Further in this section we provide
a description of the algebraic complexes that compute the homology and cohomology of the 
quotient of Salvetti complex $\Sal(\mH_W)$ by the action of the group $W.$

\begin{thm}[\cite{salv87}]
The complement $M(\mH)$ has the homotopy type of a CW-complex $\Sal(\mH)$ that is a deformation retract of $M(\mH).$
The $k$-cells of the complex $\Sal(\mH)$ are in $1$ to $1$ correspondence with the couples $(C,F)$ where 
$C$ is a chamber of the arrangement and $F$ is a codimension $k$ facet adjacent to the cell $C.$
\end{thm}

If the arrangement $\mH$ is the reflection arrangement of a Coxeter group $W,$ the complex $\Sal(\mH)$ 
is $W$-invariant and the homotopy that gives the retraction from the space $M(\mH)$ to the complex $\Sal(\mH)$ can 
be chosen to be $W$-equivariant. Furthermore, the action on the cells follows from the action of $W$ on the sets of 
chambers and facets. Fix a fundamental chamber $C_0$ for the arrangement $\mH_W.$

\begin{thm}[\cite{salvetti94, decsal96}]
Let $W$ be a Coxeter group. The orbit space $N(W)$ has the same 
homotopy type of the CW-complex $X_W = \Sal(\mH_W)/W.$

The $k$-cells of the complex $X_W$ are in $1$ to $1$ 
correspondence with the facets of $\mH_W$ that are adjacent 
to the fundamental chamber $C_0.$
\end{thm}


Let $(W,S)$ be the Coxeter system associated to the Coxeter group $W$ and to the fundamental chamber $C_0.$ 
Let $\Gamma$ be the corresponding Coxeter graph. We recall that the nodes of $\Gamma$ are in bijection with the 
elements of $S.$ Since the arrangement $\mH_W$ is locally finite, the facets of the 
arrangement $\mH_W$ that are adjacent to the fundamental chamber $C_0$ are in bijection with the finite parabolic
subgroups of $W$ generated by subsets of $S.$

\begin{cor}[\cite{salvetti94, cms08, cd95}]
Let $(W,S)$ be a Coxeter system. The $k$-cells of the complex $X_W$ are in $1$ to $1$ correspondence with the 
$k$-subsets of $S$ that generate finite parabolic subgroups.
\end{cor}

\begin{exm}
In Figure \ref{fig:B3} there is a picture of the complex $X_W$ for the symmetric group $W=\mathfrak{S}_3$
with set of generators $S = \{s_1, s_2\}.$
The $6$ vertices of the hexagon are all identified to a single vertex corresponding to the empty subset of $S$.
The $6$ edges of the hexagon are identified according to the arrows and correspond to the subsets $\{s_1\}$ and $\{s_2\}.$ The $2$-cell corresponds to the set $S$ itself. 
The complex $X_W$ is homotopy equivalent to the configuration space $C(\C,3).$

\begin{figure}[htb]
 \centering
 \includegraphics{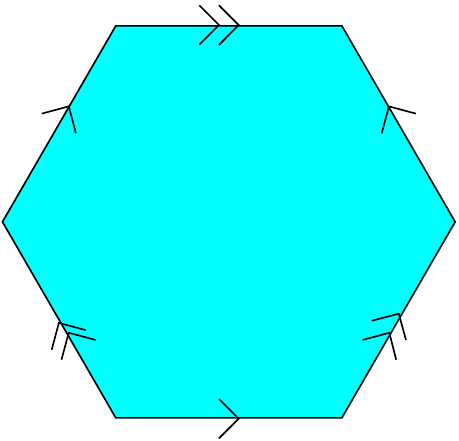}
\caption{} \label{fig:B3}
\end{figure}
 
\end{exm}

In order to provide a complete description of the complexes $\Sal(W)$ and $X_W$ for a given Coxeter system $(W,S)$ 
we need to show how the cells glue together. 
We refer the reader to \cite{salv87} and \cite{salvetti94} (see also \cite{paris}) for this. 
Here we recall the description of
the boundary map for the 
cochain complex of $X_W$ with coefficients in
an assigned local system. 
Let $M$ be a $\Z$-module and let
$$
\lambda: A_\Gamma \to \Aut(M)
$$
be a representation of the fundamental group of $X_W.$ Such a representation determines a local system $\mL_\lambda$
on the complex $X_W.$
Moreover let $(\mC^*, \delta)$ be the algebraic complex associated to the CW-complex $X_W$ that computes the cohomology
$H^*(X_W;\mL_\lambda).$ 
The complex $\mC^*$ is given by a direct sum of some copies of the $\Z$-module $M$ indexed by elements $e_T$
\begin{equation}\label{e:complex}
\mC^k := \bigoplus M.e_T 
\end{equation}
where the sum goes over all the subset $T \subset S$ such that $\mid \! T \! \mid = k$ and the parabolic 
subgroup $W_T$ is finite. The complex $\mC^*$ is graded with $\deg e_T = \mid \! T \! \mid.$

In order to define the differential $\delta$ we recall some well known facts about Coxeter groups and Artin groups.
The first result we need is the following one (see for example Proposition 1.10 in \cite{hump}).
\begin{prop} \label{p:decomp}
Let $(W,S)$ be a Coxeter system with length function $l.$ Any 
element $w \in W$ can be written in a unique way as a product
$w = uv$ with $v \in W_T$ and $u \in \underline{w} \in W/W_T$ 
such that $l(w) = l(u) + l(v).$ \end{prop}
The element $u$ is the unique 
element of minimal length 
in the coset $\underline{w} \in W/W_T$ and it is called 
the \emph{minimal coset representative} of $\underline{w}.$

Given a Coxeter system $(W_\Gamma,S),$ with Coxeter graph $\Gamma,$ and the associated Artin group $A_\Gamma,$ there
is a natural epimorphism $\pi: A_\Gamma \twoheadrightarrow W_\Gamma$ defined by mapping each standard generator 
$g_s$ of $A_\Gamma$ to the corresponding element $s\in W_\Gamma$ for all $s \in S.$ Matsumoto proves the following 
lemma (see also \cite{tits}):
\begin{lem}[\cite{mats}]
Let $(W_\Gamma,S)$ be a Coxeter system. Given an element $w \in W$ expressed as a positive word 
$s_{i_1} \cdots s_{i_l}$ of minimal length $l$ in the generators $s_j \in S,$
the corresponding element $g=g_{s_{i_1}} \cdots g_{s_{i_l}} \in A_\Gamma$ is well defined and does not 
depend on the choice of the word representing $w.$
\end{lem}  
As a consequence the map $\pi$ has a natural set-theoretic section $\psi: W \to A_\Gamma.$ We remark that the 
section $\psi$ defined according to the previous lemma is not a group homomorphism.

Let $<$ be a total ordering on the set $S.$ We can define the coboundary map $\delta$ as follows: 
for a generator $e_T \in \mC^*$ and an element $a \in M$ we have
\begin{equation} \label{e:delta1}
\delta (a.e_T) := \sum_
{
s \in S \setminus T, \mid \!W_{T \cup \{s\}} \! \mid < \infty 
}
(-1)^{\sigma(s,T)+1} \sum_{\underline{w} \in W_{T \cup \{ s\}}/W_T} 
(-1)^{l(w)}\lambda(\psi(w))(a).e_{T \cup \{s\}}
\end{equation}
where $w$ is the minimal length representative of the coset $\underline{w} \in W_{T \cup \{ s\}}/W_T$ 
and $\sigma(s,T)$ is the number of elements of the set $T$ that are strictly smaller than $s$ with
respect to the order $<.$

\begin{thm}[\cite{salvetti94}] 
Let $\mL_\lambda$ be the local system induced on the space $N(W)$ 
by a representation $\lambda$ of the group $A_\Gamma$
on the $\Z$-module $M.$ Let $(\mC^*, \delta)$ be  the complex defined by formulas (\ref{e:complex}) 
and (\ref{e:delta1}) above for the group $W = W_\Gamma.$ We have the following isomorphism:
$$
H^*(\mC^*) = H^*(N(W); \mL_\lambda).
$$
\end{thm}

We recall the following fundamental result.
\begin{thm}[\cite{deligne}]
If $W$ is a finite linear reflection group, then $N(W)$ is aspherical.
\end{thm}

As a consequence if $W$ is finite the space $N(W)$ is a classifying space 
for $A_\Gamma$ and we have an isomorphism
$$
H^*(N(W); \mL_\lambda) = H^*(A_\Gamma; M_\lambda)
$$
where $M_\lambda$ is the $\Z$-module $M$ considered as 
a $A_\Gamma$-module through the representation $\lambda.$

\subsection{Abelian representations and Poincar\'e series}
We focus now on abelian representations of $A_\Gamma$ since in that case the expression 
of formula (\ref{e:delta1}) became very simple.

\begin{rmk}
We recall how to compute the abelianization 
$A_\Gamma^{\Ab}: = A_\Gamma/[A_\Gamma, A_\Gamma]$
of the group $A_\Gamma.$ 
For a given Coxeter graph $\Gamma$ we consider the graph 
$\overline{\Gamma}$ with vertices set $S,$ the set of vertices 
of $\Gamma$ and with an edge $e_{s,t}$ for the couple $(s,t)$ if and 
only if the element $m(s,t)$ in the Coxeter matrix is odd. 
The abelianization $A_\Gamma^{\Ab}$ is the free abelian group 
generated by the connected components of the graph $\overline{\Gamma}.$ 
The abelianization map $\Ab: A_\Gamma \to A_\Gamma^{\Ab}$ maps each 
standard generator $g_s \in A_\Gamma$ to the generator corresponding 
to the connected component of the graph $\overline{\Gamma}$ 
containing the vertex $s.$
\end{rmk}

If $\lambda$ is an abelian representation, then $\lambda$ factors through 
the abelianization map $\Ab$ and the elements in the image of $\lambda$ commute.

Given a subset $H \subset W$ we define the sum $$H_\lambda := \sum_{w \in H} \lambda(\psi(w)).$$
In particular, given a subset $T \subset S$ that generates the parabolic subgroup $W_T,$ 
we call the sum $(W_T)_\lambda$ the 
\emph{Poincar\'e series  of the group $W_T$ with coefficients in the representation $\lambda$}.

As a consequence of Proposition \ref{p:decomp} we obtain the following formula:
$$(W_{T})_\lambda	
\sum_{\underline{h} \in W_{T \cup \{s\}}/W_T} \lambda(\psi(h)) 
= (W_{T \cup \{s\}})_\lambda$$
where $h$ is the minimal coset representative of $\underline{h} \in W_{T \cup \{s\}}/W_T.$

\begin{exm} \label{ex:rank_one}
We define a representation $\lambda(q): A_\Gamma \to \Aut(L)$, 
where $L = R[q^\pmu]$ is a Laurent polynomial ring with coefficients in a ring $R$ and
$\lambda(q)(g_s)$ is multiplication by $q$ for each standard generator of $A_\Gamma.$
In this case the series $W(q):=W_{\lambda(q)}$ is called the 
\emph{Poincar\'e series} for $W.$
From formula (\ref{e:delta1}) we get
\begin{equation} \label{e:delta2}
\delta (a.e_T) := \sum_
{
s \in S \setminus T, \mid \!W_{T \cup \{s\}} \! \mid < \infty 
}
(-1)^{\sigma(s,T)+1} 
\frac{W_{T \cup \{s\}}(-q)}{W_T(-q)}
.e_{T \cup \{s\}}
\end{equation}
If $W$ is a finite Coxeter group with exponents 
$m_1, \ldots, m_n$ the Poincar\'e series is actually a polynomial 
and the following product formula holds (\cite{sol66}):
$$W(q) = \prod_{i=1}^n (1+q+\cdots+q^{m_i}).$$
\end{exm}

\begin{exm}
An analog of Example \ref{ex:rank_one} is given by a representation 
on the Laurent polynomial ring in two variables  $L = R[q_1^\pmu, q_2^\pmu].$
Let $\Phi$ be a root system with two different root-lengths. 
As an example consider the root systems of type $\mathbb{B}_n$ or any reducible root system. 
Let $W$ be the Coxeter group associated to the root system $\Phi.$
We can define a representation of $W$ on the ring $L$ as follows: if $\alpha$ is a short
root and $s$ is the reflection associated to $\alpha \in \Phi$ the generator $g_s$ maps to 
multiplication by $q_1$ and if $t$ is the reflection associated to a long root $\beta \in \Phi$
$g_t$ maps to multiplication by $q_2.$ The Poincar\'e series for $W_{\B_n}$ with coefficients in such
a representation are computed in \cite{reiner}.
\end{exm}

\begin{exm} \label{ex:sal_b3} We show an explicit computation of the cochain complex $\mC^*$ 
and we compute the coboundary $\delta$ in the case of the Coxeter group $W = W_{\A_2} = \mathfrak{S}_3,$
with coefficients in the local system $\mL_\lambda = \Z[q^\pmu]$ given as in Example \ref{ex:rank_one}.
The complex that we are going to describe computes the cohomology 
of the commutator subgroup of the braid group $\Br_3,$ up to a degree shift (see Theorem \ref{t:shifting}):
$$
H^*(\mC^*)= H^*(\Br_3; \Z[q^\pmu]_\lambda) = H^{*+1}(\Br_3'; \Z).
$$
We recall that the set of standard generators for the group $W$ is $S = \{s_1,s_2 \}.$ Hence the complex $\mathcal{C}^*$ is given by
$$\mathcal{C}^0= \Z[q^\pmu].e_{\varnothing};$$
$$\mathcal{C}^1= \Z[q^\pmu].e_{\{s_1\}} \oplus \Z[q^\pmu].e_{\{s_1\}};$$
$$\mathcal{C}^2= \Z[q^\pmu].e_{\{s_1, s_2\}}.$$
According to the formulas in Example \ref{ex:rank_one}, the Poincar\'e series are given by
$$
W_{\varnothing}(q) = 1; 
$$
$$
W_{\{s_1\}}(q) = W_{\{s_1\}}(q) = 1-q;
$$
$$
W_{\{s_1, s_2\}}(q) = (1-q)(1-q+q^2)
$$
and hence the coboundary is
$$
\delta e_{\varnothing} = (1-q) e_{\{s_1\}} + (1-q)e_{\{s_2\}} 
$$
$$
\delta e_{\{s_1\}} = - \delta e_{\{s_2\}} = (1-q+q^2) e_{\{s_1,s_2\}}.
$$
\end{exm}

\begin{rmk}
The analog construction of the algebraic complex $(\mC^*, \delta)$ can be given for homology.
We have a complex
\begin{equation}\label{e:complex_homology}
\mC_k:= \bigoplus_{\mid \!T \! \mid = k, \mid \!W_{T} \! \mid < \infty} M.e_T 
\end{equation}
with  boundary maps
\begin{equation} \label{e:delta_homology}
\partial (a.e_T) := \sum_
{
s \in T
}
(-1)^{\sigma(s,T)+1} \sum_{\underline{w} \in W_{T }/W_{T\setminus \{ s\}}} 
(-1)^{l(w)}\lambda(\psi(w))(a).e_{T \setminus \{s\}}
\end{equation}
so that $H_*(\mC_*) = H_*(N(W); \mL_\lambda).$ 
\end{rmk}





\section[Filtrations and spectral sequences]{Filtrations and spectral sequences for the Salvetti complex} 
\label{s:filtration}

\subsection{A natural filtration for the Salvetti complex} \label{ss:natural}
In this section we assume that we have a Coxeter graph 
$\Gamma$ with \emph{finite} set of vertices $S$ and 
a corresponding Coxeter group $W=W_\Gamma$ and a 
Coxeter system $(W,S).$ We fix an ordering $<$ on $S$ 
and we assume $S= \{s_1, \cdots, s_N\},$ with $s_1 < \cdots < s_N.$
Moreover we set a $\Z$-module $M$ and a representation 
$\lambda: A_\Gamma \to \Aut(M).$

The ordering on the set $S$ induces a natural decreasing 
filtration on the complex $\mC^*$ defined in Section \ref{s:salvetti}. 
We define the submodule
$$
\mF^k\mC^* := <e_T \mid s_{N-k+1}, \cdots, s_N \in T>.
$$
It is clear from the description of the differential $\delta$ 
(see equation (\ref{e:delta1})) that the submodule 
$\mF^k\mC^*$ is a subcomplex of the complex $(\mC^*, \delta)$ and  we have the inclusions

$$0 = \mF^{N+1}\mC^* \subset \cdots \subset \mF^{k+1}\mC^* \subset \mF^k\mC^* \subset \cdots \subset \mF^0\mC^* = \mC^*.$$

By standard methods (see for example \cite{spa}) we have a spectral sequence associated 
to the complex $(\mC^*, \delta)$ and the filtration $\mF$:
\begin{thm} \label{t:ss}
There is a first-quadrant spectral sequence $(E_r,d_r)$ with $E_0$-term
$$
E_0^{i,j} = \mF^i\mC^j/\mF^{i+1}\mC^j \Longrightarrow H^{i+j}(\mC^*).
$$ 
The $d_0$ differential is the map naturally induced by the differential $\delta$ on the
quotient complex $\mF^i\mC^{i+j}/\mF^{i+1}\mC^{i+j}.$ 
The $E_1$-term of the spectral sequence is given by
$$
E_1^{i,j} = H^{i+j}(\mF^i\mC^*/\mF^{i+1}\mC^*)
$$
and the $d_1$ differential corresponds to the boundary operator of the triple 
$(\mF^{i+2}C^j, \mF^{i+1}C^j, \mF^{i}C^j).$
\end{thm}

\begin{exm}

In the case of the complex $(\mathcal{C}^*, \delta)$ of Example \ref{ex:sal_b3} ($W=W_{\A_2}$) 
the filtration gives a very easy picture.
The term $\mF^0\mC^*$ is the complex $\mC^*$ itself. The term $\mF^1\mC^*$ is the $\Z[q^\pmu]$-submodule
generated by $e_{\{s_2\}}$ and $e_{\{s_1, s_2\}}.$ The term $\mF^2\mC^*$ is the submodule generated by 
$e_{\{s_1, s_2\}}.$ Finally $\mF^3\mC^*$ is the trivial submodule.
It is easy to see that the quotient $\mF^0\mC^*/\mF^1\mC^*$ is isomorphic to the complex 
$(\mathcal{C}_{\A_1}^*, \delta)$ for $W=W_{\A_1} = \mathfrak{S}_2$ (recall that the 
corresponding Artin group is the braid group $\Br_2 = \Z$), with the correspondence
$$
\iota:\mF^0\mC^*/\mF^1\mC^* \to \mathcal{C}_{\A_1}^*
$$
given by $\iota:[e_{\{s_1\}}] \mapsto e_{\{s_1\}}$ and $\iota:[e_{\varnothing}] \mapsto e_{\varnothing}$.
It is easy to verify that the isomorphism $\iota$ is compatible with the coboundary map $\delta$. Moreover, note that
$\iota$ preserves the natural graduation. We assume that the ring of coefficients $\Z[q^\pmu]$
is naturally graded with degree zero.
The quotient $\mF^1\mC^*/\mF^2\mC^*$ (resp. $\mF^2\mC^*/\mF^3\mC^*$) is 
isomorphic, as a $\Z[q^\pmu]$-module, to $\Z[q^\pmu]$ generated by $[e_{\{s_2\}}]$ 
(resp. $[e_{\{s_1, s_2\}}]$) with graduation shifted by $1$ (resp. $2$). 
Let $\lambda$ be the representation defined in Example \ref{ex:rank_one}. Note that 
$\lambda$ is compatible with the natural inclusion $\Br_m \into \Br_{m+1}$.
Hence we can write the $E_1$-term of the spectral sequence associated to $(\mC^*, \delta)$ as follows
 \begin{center}
 \begin{tabular}{|l}
 \xymatrix @R=1pc @C=1pc {
H^1(\Br_2; \Z[q^\pmu]_\lambda) & 		      & 			\\
H^0(\Br_2; \Z[q^\pmu]_\lambda) & \Z[q^\pmu]_\lambda & \Z[q^\pmu]_\lambda }\\
\hline
 \end{tabular}
 \end{center}

\end{exm}

\subsection{The differentials} The differentials of the spectral sequence given in 
Theorem \ref{t:ss} are induced by 
the coboundary $\delta$ of the complex $\mC^*.$ 
The differential $d_1$ is explicitly described in Theorem \ref{t:ss}. In order to compute the higher differentials
it is useful to control the representatives in $\mC^*$ for the elements of the spectral sequence.

Following the construction in \cite{spa} we define $Z^s_r:= \{c \in \mF^s\mC^* \mid \delta c \in \mF^{s+r}\mC^* \}.$
Given an element $x \in E_r,$ it is represented by a cochain 
$$c  \in {Z^s_r}/{(Z^{s+1}_{r-1} + \delta Z^{s-r+1}_{r-1})},$$
hence by a class $\overline{c} \in  \mF^s\mC^*$ such that 
$\delta \overline{c} \in \mF^{s+r}\mC^*$
modulo the subgroup $$(\delta \mF^{s-r+1}\mC^* \cap \mF^s\mC^*) + \mF^{s+1}\mC^*.$$

The differential $d_r$ on the class $x$ is the map induced by the coboundary $\delta.$
Hence, given 
$b \in {Z^{s+r}_r}/{(Z^{s+r+1}_{r-1} + \delta Z^{s+1}_{r-1})}$ and 
$\overline{b} \in \mF^{s+r}\mC^*/\mF^{s+r+1}\mC^*$ representatives of an element $y \in E_r,$  if $d_r x = y $ 
we have that $\delta c - b \in (Z^{s+r+1}_{r-1} + \delta Z^{s+1}_{r-1})$ and, if $y=0,$ 
$c \in Z^s_{r+1} + Z^{s+1}_{r-1}.$ In an equivalent way we can say that 
$d_rx = b$ if and only 
if $\delta \overline{c} - \overline{b} \in (\mF^{s+r+1}\mC^* + \delta \mF^{s+1}\mC^*).$

Given an element $x \in E_r$ such that $d_r x=0$ we need to lift $x$ to an element $x' \in E_{r+1}.$ We
begin taking a representative $c \in Z^s_{r+1} + Z^{s+1}_{r-1}$ for $x$ and we
choose
a lifting $c' \in Z^s_{r+1}$ 
with $c' = c + \Delta,$ where $\Delta \in Z^{s+1}_{r-1}.$
This means that we need to lift the class $\overline{c}$ to a class 
$\overline{c}' \in \mF^s\mC^*/((\delta \mF^{s-r+1}\mC^* \cap \mF^s\mC^*) + \mF^{s+1}\mC^*)$
taking as a representative for $\overline{c}'$ the element $\overline{c} + \overline{\Delta}$ where 
$\overline{\Delta} \in \mF^{s+1}\mC^*$ and $\delta (\overline{c} + \overline{\Delta}) \in \mF^{s+r+1}\mC^*.$

Working out the spectral sequence we can use Theorem \ref{t:ss} and start at page $E_1$ choosing 
a class $x \in H^*(\mF^s\mC^*/\mF^{s+1}\mC^*)$ and a representative $c_1 \in \mF^s\mC^*$ for $x.$
At the $E_r$-step of the spectral sequence we have a representative $c_r$ for $x$ 
with $\delta c_r \in \mF^{s+r}\mC^*$
and if $d_rc_r =0$
we can choose in $E_{r+1}$ a new representative $c_{r+1} = c_r + \Delta_r$ with 
$\Delta_r \in \mF^{s+1}\mC^*$ and 
$\delta (c_r + \Delta_r) \in \mF^{s+r+1}\mC^*.$

\subsection{Recursion and order of vertices} \label{ss:recursion}
Thanks to the simple structure of the complex $\mC^*$ and the filtration $\mF^*,$ Theorem \ref{t:ss}
can provide a recursive description of the cohomology of the complex $C^*$ and hence of the space $N(W).$
The Coxeter graph of the group $W$ as well as the choice of the ordering
on the set $S$ of vertices of $\Gamma$ play an important role in this. 

Let $\Gamma_{\overline{k}}$ be the full subgraph of $\Gamma$ with vertices $s_1, \ldots, s_{N-k-1}$
and let $\Gamma_{\widetilde{k}}$ be the full subgraph of $\Gamma$ with vertices $s_{N-k+1}, \ldots, s_{N}.$

\begin{prop}\label{p:recursion}
Let $A_\Gamma$ be the Artin group associated to the Coxeter graph $\Gamma.$
Suppose that the parabolic subgroups associated to the graphs $\Gamma_{\overline{k}}$  
and $\Gamma_{\widetilde{k}}$ commute, i.~e.~for every vertex
$s \in s_{1}, \ldots, s_{N-k-1}$ and $t \in a_{N-k+1}, \ldots, s_{N}$ we have $m(s,t) = 2.$ Then the quotient complex
$\mF^k\mC^*/\mF^{k+1}\mC^*$ is isomorphic to the complex $\mC^*(\Gamma_{\overline{k}})[k],$ 
that is the cochain complex that computes the cohomology of the Artin group 
$G_{\Gamma_{\overline{k}}}$ with a graduation shifted by $k.$

The isomorphism $$\mC^*(\Gamma_{\overline{k}})[k] \stackrel{\rho}{\longrightarrow} \mF^k\mC^*/\mF^{k+1}\mC^*$$
is defined as follows: given a subset $T \subset \{s_1, \ldots, s_{N-k-1}\},$ 
the generator $e_{T}$  maps to the equivalence class of the generator $e_{T'},$ with
$T' = T \cup \{ s_{N-k+1}, \ldots, s_N \}.$
\end{prop}

\begin{rmk} \label{rm:linear}
In the special case when the Coxeter graph $\Gamma$ is a subgraph of a linear graph we can 
sort the the vertices of $\Gamma$ in linear order, in such a way that for every index $k$ the hypothesis of 
Proposition \ref{p:recursion} hold. Choose such an ordering for the vertices of 
$\Gamma.$ Hence, according to Theorem \ref{t:ss}, the construction described above determines 
a spectral sequence $(E_r,d_r)$ converging to the cohomology of the Artin group $A_\Gamma.$ The 
recursion given
by Proposition \ref{p:recursion} implies that for every $i$, the $i$-th column of the $E_1$-term 
of the spectral sequence is isomorphic to the cohomology of the Artin group $A_{\Gamma'}$ for $\Gamma'$
a subgraph of $\Gamma.$
\end{rmk}

\begin{exm} \label{ex:A_4}
We keep working with a generic $\Z$-module $M$ and a representation 
$\lambda: A_\Gamma \to \Aut(M)$ as in Section \ref{ss:natural}, but we consider the special case of 
the Coxeter group $W$ of type $\A_4,$ with diagram
\begin{mucca}
\entrymodifiers={=<4pt>[o][F-]}
\begin{center}
\begin{tabular}{l}
\xymatrix @R=2pc @C=2pc {
 \ar @{-}[r]_(-0.20){s_1} &  \ar @{-}[r]_(-0.20){s_2} &  \ar @{-}[r]_(-0.20){s_3} \ar @{-}[r]_(0.80){s_4} & 
}
\end{tabular}
\end{center}\end{mucca}
and with the order on vertices given by the labelling. It is clear from the diagram that the given order satisfies the hypothesis of Proposition \ref{p:recursion}. 
Moreover we have the following isomorphisms:
$$
\mF^0\mC^*/\mF^1\mC^* = \mC^*(\A_3); \; \;\;\;\; \mF^1\mC^*/\mF^2\mC^* = \mC^*(\A_2)[1]; \; \;\;\;\;  
\mF^2\mC^*/\mF^3\mC^* = \mC^*(\A_1)[2]; $$ 
$$\mF^3\mC^*/\mF^4\mC^* = M[3]; \; \;\;\;\; \mF^4\mC^*/\mF^5\mC^* = \mF^4\mC^* = M[4]  
$$
where the index $[k]$ in square brackets means that the graduation 
of the module is shifted by $k.$
Note that in this example the Artin group $A_\Gamma$ is the braid group on $5$ strands.
According to \cite{paris} we write $\Br_i$ for the braid group on $i$ strands.
We consider the natural identification of the groups $\Br_i,$ $i <5$ as subgroups of $\Br_5$
through the diagram inclusion induced by the filtration. Hence in this case we identify
$\Br_i$ with the subgroup generated by $s_1, \ldots, s_i.$ We  keep using the notation $\lambda$
for the representation of the subgroups of $\Br_5$ induced by the inclusion.
The cohomology $H^*(N(W_{\A_4});\mL_\lambda)$ - that is the cohomology $H^*(\Br_5; M_\lambda)$ 
of the classical braid group
$\Br_5$ on $5$ strands with coefficients 
on the $\Br_5$-module $M_\lambda$ 
- can be computed 
by means of a spectral sequence with the following $E_1$-term:
 \begin{center}
 \begin{tabular}{|l}
 \xymatrix @R=1pc @C=1pc {
H^3(\Br_4; M_\lambda) & 		      & 			&   &\\
H^2(\Br_4; M_\lambda) & H^2(\Br_3; M_\lambda) & 			&   &\\
H^1(\Br_4; M_\lambda) & H^1(\Br_3; M_\lambda) & H^1(\Br_2; M_\lambda) 	&   &\\ 
H^0(\Br_4; M_\lambda) & H^0(\Br_3; M_\lambda) & H^0(\Br_2; M_\lambda) 	& M & M 
 }\\
\hline
 \end{tabular}
 \end{center}
The cohomology of the groups $\Br_i$ for $i <5$ (and actually for any $i$) can be computed recursively by means of an
analog spectral sequence.
\end{exm}

\begin{rmk}\label{rm:homology}
In the homology complex $\mC_*$ the dual filtration is given by
$$
\mF_k\mC_* := < e_T \mid \{ s_{N-k+1}, \ldots , s_N \} \varsubsetneq T>.
$$ 
With the hypothesis of Proposition \ref{p:recursion} we have that the quotient 
$\mF_{k+1}\mC_* / \mF_{k}\mC_*$ is isomorphic to the complex $\mC_*(\Gamma_{\overline{k}})[k].$
\end{rmk}

If $\mH$ is a finite central arrangement we can associate to every hyperplane $H \in \mH$ a linear functional
$l_H$ with $\ker l_H = H.$ The homogeneous polynomial $f_\mH = \prod_{H \in \mH} l_H,$ 
which is unique up to multiplication
by an invertible element, is the \emph{defining polynomial} of the arrangement 
and the set $f_\mH^{-1}(1)$ is the \emph{Milnor fiber} of the arrangement 
(see \cite{miln} for a general introduction).
If $\mH = \mH_W$ is the reflection arrangement of a Coxeter group $W,$ 
the polynomial $f^2_\mH = \prod_{H \in \mH}l_H^2$
is $W$-invariant and hence defines a \emph{weighted homogeneous} polynomial $\phi:V/W \to \C$ 
on the affine variety $V/W$ with non-isolated singularity $\phi^{-1}(0) = (\cup_{H\in \mH}H)/W.$ 
The map $\phi$ restricts to a fibration $\phi: N(W) \to \C^*$ with 
fiber $F_W = \phi^{-1}(1)$ that is called the Milnor fiber of the singularity
associated to $W.$

Let $\lambda(q)$ be the representation on the Laurent polynomial 
ring $L=R[q^\pmu]$ considered in Example \ref{ex:rank_one}. 
Let $W$ be a finite Coxeter group, with Coxeter graph $\Gamma.$ 
The fibration $\phi: N(W) \to \C^*$ induces on fundamental groups a map $\phi_\sharp:A_\Gamma \to \Z$
sending each standard generator of the Artin group to $1.$ 
Since the space $N(W)$ is aspherical, from Shapiro's Lemma (see \cite{brown}) we have that
the cohomology of the Milnor fiber $F_W$ with constant coefficients in the ring $R$ 
is isomorphic to the cohomology of $N(W)$ with coefficients in the $A_\Gamma$-module
of Laurent series $R[[q^\pmu]]$ where each standard generator of $A_\Gamma$ maps
to multiplication by $q:$
$$H^*(F_W; R) = H^*(A_\Gamma; R[[q^\pmu]]).$$
Using the recursive description of the spectral sequence for the Salvetti complex, 
in \cite{cal05} it is shown that the cohomology of the Artin group $A_\Gamma$ with coefficients in the 
representation $\lambda(q)$ is isomorphic, modulo an index shifting, to the cohomology with constant 
coefficients of the Milnor fiber $F_W.$ We can state the result as follows:
\begin{thm}[\cite{cal05}] \label{t:shifting}
Let $W$ be a finite Coxeter group and let $A$ the associated Artin group. We have:
$$
H^{*+1}(A;\mL_q) = H^*(F_W;R).
$$
\end{thm}


\begin{rmk}
A recursive computation applies even if $\Gamma$ is not a linear graph or 
if the order on the set of vertices is not linear.
For any subset $T$ of the set $S$ of vertices of $\Gamma$ we can define the following subcomplex of $\mC^*:$
$$\mF^{T}\mC^*:= < e_{U} \mid T \subset U \subset S>.$$
We can consider the poset 
$$\mathcal{P}:= \{(T,T') \mid T \subset T' \subset S  \}$$
with the order relation given by $(T_1,T_1') < (T_2, T_2')$ if and only if $T_1 \subset T_2,$ $T_1' \subset T_2'.$
Given a couple $(T,T') \in \mathcal{P} $ the recursive method  
described in this section allows one to compute the $E_1$-term of 
the spectral sequence for the cohomology of the quotient complex 
$\mF^{T}\mC^*/\mF^{T'}\mC^*$ by recursion on the poset $\mathcal{P}.$

%
%
\end{rmk}

In the next section we present a few examples of the application of this method and some results obtained with it.








\section{Cohomology of Artin groups: some examples} \label{s:examples}

In this section we recall some computations and examples where the methods from the previous section apply.
In some cases, like in Section \ref{ss:classical} and Section \ref{ss:rat},
the use of the spectral sequence described 
in Section \ref{s:filtration} makes computations and proof shorter.

In what follows we will use sometimes a compact notation for the generators $e_T,$ $T \subset S$ of the $L$-module
$\mC^*$ for the Coxeter system $(W,S).$ If $S$ is the ordered set $\{s_1, \ldots, s_n \}$ we will write a string 
$\epsilon_1\epsilon_2\cdots\epsilon_n,$ $\epsilon_i \in \{0,1 \}$ for a generator $e_T$ such
that $\epsilon_i = 1$ if and only if $s_i \in T.$   We will write also $0^h$ and $1^h$ instead of
$\underbrace{0 \cdots 0 }_{h \mbox{ terms}}$ and $\underbrace{1 \cdots 1 }_{h \mbox{ terms}},$ 
meaning respectively $e_{\empty}$ and $e_S.$ As an example we write $1^n$ for $e_S$ and $10^{n-1}$ 
for $e_{\{s_1\}};$ we can also use notations like $A01^2$ to denote a the terms
$e_T$ such that $s_{n-2} \notin T,$ $s_{n-1} \in T,$ $s_n \in T.$

\subsection{Homology of the braid group $\Br_n \! \mod 2$}\label{ss:classical}
In the case of the classical braid group $\Br_n$ with constant coefficients it is more simple to compute 
the homology instead of the cohomology. We state in this form the results obtained by Fuks in \cite{fuks}
and we give a somewhat simpler proof.

\begin{thm}
The homology $\oplus_n H_*(\Br_n; \Z_2)$ is isomorphic to the ring $R = \Z_2[x_0, x_1, x_2, \cdots]$ 
considered as a $\Z_2$-module. The variable
$x_i$ has homological dimension $\dim x_i = 2^i-1$ and degree $\deg x_i= 2^i$ so that
the monomial $x_{i_1}^{h_1}\cdots x_{i_l}^{h_l}$ belongs to the homology group
$H_m(\Br_n, \Z_2)$ with $n = \sum_j h_j 2^{i_j}$ and $m = \sum_j h_j (2^{i_j}-1).$
The multiplication map $\Br_{n_1} \times \Br_{n_2} \to \Br_{n_1+n_2}$ given by juxtaposing braids
induces a multiplication on $\oplus_n H_*(\Br_n; \Z_2)$ that corresponds to the standard multiplication
 in the ring $R.$
\end{thm}
\begin{proof}
We consider the constant local system $\mL = \Z_2$ where 
each standard generator acts by multiplication by $1.$
Using the notation of Example \ref{ex:rank_one} we set $q = -1.$ The coefficients in the
boundary $\partial$ can be easily computed since $1 + q + \cdots + q^{n-1} = n \mod 2.$
In particular we have that the boundary for a simple element in the form $c= 1^{n-1}$ is
given ($\!\!\!\!\mod 2$) by
$$
\partial c = \sum_{i=1}^{n-1} \bin{n}{i} 1^{i-1}01^{n-i-1}.
$$
We recall that the binomial $\bin{n}{i}$ is even if and only if the integers $i$ and $n-i$ 
have no common non-zero coefficients in their expansion in base $2.$ As a special case we have that 
if $n$ is a power of $2$ then the  binomial $\bin{n}{i}$ is always even.

%
Given a monomial $u = x_{i_1}^{h_1} \cdots x_{i_l}^{h_l}$ we assume that the indexes of $u$ are ordered 
$i_1 > i_2 > \cdots > i_l$ and we associate to $u$ the following generator 
in the Salvetti complex $\mC_*$ for $\Br_n$:
$$
\underbrace{1^{2^{i_1}-1}0\cdots01^{2^{i_1}-1}}_{h_1 \mbox{ terms}} 0 \cdots 0 
\underbrace{1^{2^{i_l}-1}0\cdots01^{2^{i_l}-1}}_{h_l \mbox{ terms}}.
$$
From the description of the boundary map 
$\partial$  it follows that for any generator of $\mC_*$ in the form
$$
c = 1^{2^{a_1}-1}0\cdots01^{2^{a_l}-1}
$$
we have that $\partial c=0$ and then in particular all 
the generators associated to monomials in $R$ are cycles.

Moreover given two generators
$$
c_1 = 1^{2^{a_1}-1}0\cdots01^{2^{a_l}-1}01^{2^b-1}01^{2^{b'}-1}01^{2^{a'_1}-1}0\cdots01^{2^{a'_{l'}}-1}
$$
and
$$
c_2 = 1^{2^{a_1}-1}0\cdots01^{2^{a_l}-1}01^{2^{b'}-1}01^{2^{b}-1}01^{2^{a'_1}-1}0\cdots01^{2^{a'_{l'}}-1}
$$
we can set
$$
c = 1^{2^{a_1}-1}0\cdots01^{2^{a_l}-1}01^{2^b+2^{b'}-1}01^{2^{a'_1}-1}0\cdots01^{2^{a'_{l'}}-1}
$$
and for $b \neq b'$  we have $\partial c = c_1 + c_2.$ Hence the two cycles $c_1$ and $c_2$ are co-homologous.

We assume the inductive hypothesis that for any $k < n$ the 
cycles corresponding to the monomials 
with total degree $k$ generate the homology
group $H_*(\Br_k; \Z_2).$ Using the filtration given in Remark \ref{rm:homology}
we can define the spectral sequence for the homology of $\Br_n$ 
analogous to the cohomology spectral sequence
constructed in Theorem \ref{t:ss}.

The $E^1$-term is given by $E^1_{s,t}= H^t(\Br_{n-s-1}; \Z_2).$ By induction the $s$-th 
column of the $E^1$-term of the spectral sequence is generated by the 
monomials in $R$ with degree $n-s-1.$ If the string $c$ is the cycle associated to the 
monomial $u \in R,$ the representative in $\mC_*$ of a monomial $u$ in the $s$-th column
of the $E^1$-term is given by the string $c01^s.$

The differential $d^1_{s,t}: E^1_{s,t} \to E^1_{s-1,t}$ acts on $c01^s$ by mapping 
$d^1: c01^s \to s \cdot c001^{s-1},$ that is the representative 
of the monomial $s\cdot ux_0 \mod 2.$ This means 
that $d^1_{s,t}$ it is given by multiplication by $sx_0$ and 
hence it is trivial if and only if $s$ is even, while it is injective for odd $s.$ 
It follows from the inductive hypothesis on the description of the groups
$H_*(\Br_k; \Z_2)$ for $k<n$ that for $s$ even we have $E^2_{s,t} =0$  and
for $s$ odd $E^2_{s,t}$ is generated by all the monomials with degree $n-s-1$ and dimension $s$ that are not 
divided by $x_0.$

The differential $d^2_{s,t}: E^2_{s,t} \to E^2_{s-2,t+1}$ is given by multiplication 
by $x_1$ if $s-1 \equiv 0 \mod 4$ and is trivial otherwise. 
The $s$-th column of the $E^3$-term of the spectral sequence is trivial 
if $s-1 \equiv 0 \mod 4,$ $s > 1$ and is generated by monomials 
that are not divided by $x_0$ and $x_1$ if $s-1 \equiv 2 \mod 4.$

In general the description of the differential, and as a consequence 
the description of the spectral sequence, is
the following. 
The differential $d^{2^i}_{s,t}: E^{2^i}_{s,t} \to E^{2^i}_{s-{2^i},t+2^i-1}$ 
is given by multiplication by $x_i$ if $s-1 \equiv 0 \mod 2^i$ and is trivial otherwise. 
The $s$-th column $E^{2^i+1}$-term of the spectral sequence is trivial 
if $s-1 \equiv 0 \mod 2^i,$ $s > 2^i$ and is generated by monomials 
that are not divided by $x_0, x_1, \ldots, x_i$ if $s-1 \equiv 2^{i-1} \mod 2^i.$ 
All the other differentials are trivial.

In the $E^\infty$-term of the spectral sequence we have,  in the  $0$-th column, the monomials
$u$ with degree $n-1.$ Those lift to monomials $ux_0$ in the homology of $\Br_n.$
In general in the $(2^i-1)$-th column we have the monomials with degree $n - 2^i$ that are not divided
by the terms $x_0, \ldots, x_{i-1}.$ A monomial $u$ in the  $(2^i-1)$-th column lifts to the monomial $ux_i$
in the homology of $\Br_n.$

%
%
The multiplication map $\Br_{n_1} \times \Br_{n_2} \to \Br_{n_1+n_2}$ given by juxtaposing braids
is induced by the inclusion of the Coxeter graph $\Gamma_{A_{n_1-1}}$ for $W_{A_{n_1-1}}$ and 
$\Gamma_{A_{n_2-1}}$ for $W_{A_{n_2-1}}$ in the graph $\Gamma_{A_{n_1+n_2-1}}$ for $W_{A_{n_1+n_2-1}}$
as graphs of commuting parabolic subgroups. The map sends the vertices of $\Gamma_{A_{n_1-1}}$ to
the first $n_1-1$ vertices of $\Gamma_{A_{n_1+n_2-1}}$ and the vertices of $\Gamma_{A_{n_2-1}}$ to 
the last $n_2-1$ preserving the ordering.
The induced map on the Salvetti complex is given by mapping the couple of strings $(A,B)$ to the string $A0B$
and hence the induced multiplication in homology maps the couple of monomials $(u,v)$ to the product $uv.$
\end{proof}
\subsection{Rational cohomology of the Milnor fiber}\label{ss:rat}
In this example we show how to compute the rational cohomology of the classical braid group $\Br_n$ 
with coefficients in the representation $\lambda(q)$ already described in Example \ref{ex:rank_one}. 
The result presented here has been computed in \cite{fren} and \cite{mar} and independently in \cite{dps}. 

Let $R:= \Q$ be the field of rational numbers and let $\mL_q$ be the local system constructed in
Example \ref{ex:rank_one}. The local system is induced by the action of the braid group on 
the Laurent polynomial ring $L:= \Q[q^\pmu].$ Each standard generator maps to multiplication
by $(-q).$ The choice of this action is clearly equivalent to the action given by each standard 
generator mapping to multiplication by $q$, as in Example \ref{ex:rank_one}. Although we prefer the choice of $(-q)$,
in coherence with \cite{dps, cal05, cal06} and others, in order to get slightly simpler formulas, 
as the reader can see in the following paragraphs.

As showed in Section \ref{ss:recursion}, this local system has an interesting geometric 
interpretation in terms of the cohomology of the Milnor fiber of
the discriminant singularity of type $\A_{n-1}$ (see also \cite{cal05, cal06} for the analog computation
for homology with integer coefficients). 
%
%
From an algebraic point of view, the computation gives, modulo an index shifting, the rational cohomology of the 
kernel of the abelianization map $\Br_n \to \Z,$ that is the commutator subgroup $\Br_n'$ of the braid group
on $n$ strands. 
In fact it is easy to see that the Milnor fiber of type $\A_{n-1}$ is a classifying space for $\Br_n'$ and using Theorem \ref{t:shifting} we get:
$$
H^{*+1}(\Br_n;\mL_q) = H^*(\Br_n';\Q).
$$

Let $\varphi_n(q)$ be the $n$-th cyclotomic polynomial. We introduce the notation 
$\mathbf{n}:= \Q[q]/(\varphi_n(q)).$ In the following paragraphs we will also use the notation
$[n]:= 1 + q + \cdots + q^{n-1}= \frac{q^n-1}{q-1}.$ 

For any positive integer $n$ we linearly order the vertices of the graph $\Gamma_{n}$ of type $\A_n,$ that is the 
graph for the Artin group $\Br_{n+1}.$ Let $\mC_n^*$ be the complex associate to $\Gamma_n.$ 
Recall that the Coxeter group $W_{\A_n}$ has exponents $1,\ldots, n$ and hence 
$W_{\A_n}(q) = [n+1]! :=\prod_{i=1}^{n+1} [i].$
From Example \ref{ex:rank_one} we can describe more explicitly the coboundary $\delta$ in $\mC^*_n.$
Let $e_T$ be a generator of $\mC^*_n$ in the form $A01^a01^b0B$ and let $e_{T'}$ be the generator
$A01^{a+b+1}0B.$ 
We need the following simple remark: if $W_S$ is a Coxeter group generated by a set of generator $S$ that
is the disjoint union $S= T_1 \cup T_2$ of two commuting set of generators, then 
we can decompose $W_S = W_{T_1} \times W_{T_2}$ and we have a factorization 
$W_S(q) = W_{T_1}(q) \times W_{T_2}(q)$ for the Poincar\'e series for $W.$
Applying this to the computation of $\delta e_T$ we have that the coefficient for $e_{T'}$ in the 
coboundary is given by the sign coefficient $(-1)^{a+|A|}$ times the $q$-analog binomial
$$
\frac{W_{\A_{a+b+1}}(q)}{W_{\A_a}(q)W_{\A_b}(q)} = \frac{[a+b+2]!}{[a+1]![b+1]!} := \qbin{a+b+2}{a+1}.
$$

As in \cite{dps} we define the following elements:
\begin{eqnarray*}
w_h & := & 01^{h-2}0\\
z_r & := & 1^{h-1}0 + (-1)^h 01^{h-1}\\
b_h & := & 01^{h-2}\\
c_h & := & 1^{h-1}\\
z_h(i) & := & \sum_{j=0}^{j=i-1} (-1)^{hj}w_h^j z_h w_h^{i-j-1}\\
v_h(i) & := & \sum_{j=0}^{j=i-2} (-1)^{hj}w_h^j z_h w_h^{i-j-2} b_h + (-1)^{h(i-1)} w_h^{i-1} c_h.
\end{eqnarray*}
We remark that the elements  $z_h(i)$ and $v_h(i)$ are cocycles.

Our aim is to prove the following result:


\begin{thm} [\cite{dps}] \label{t:dps}

\begin{eqnarray*}
H^{n-2i+1}(\Br_{n+1}; \mL_q) & = &
\left\{
\begin{array}{cl}
0 & \mbox{if } h:=\frac{n}{i} \mbox{ is not an integer} \\
\q{h} & \mbox{generated by } [z_h(i)] \mbox{ if } h:= \frac{n}{i} \mbox{ is an integer}
\end{array}
\right. \\
H^{n-2(i-1)}(\Br_{n+1}; \mL_q) & = &
\left\{
\begin{array}{cl}
0 & \mbox{if } h:=\frac{n+1}{i} \mbox{ is not an integer} \\
\q{h} & \mbox{generated by } [v_h(i)] \mbox{ if } h:=\frac{n+1}{i} \mbox{ is an integer.}
\end{array}
\right.
\end{eqnarray*}
\end{thm}
\begin{proof}
We can prove the Theorem by induction on $n.$ 
We consider 
the natural graph inclusion $\Gamma_n \into \Gamma_{n+1}$ and group inclusion $\Br_n \into \Br_{n+1}$ 
induced by the filtration $\mF.$ As in Example \ref{ex:A_4} we recall that the $E_1-$term of the spectral
sequence for $\Br_{n+1}$ is given by  
$$ E^{s,t}_1 := H^{s+t}(\mF^s\mC^*_n/\mF^{s+1}\mC^*_n) = H^t (\mC_{n-s-1})$$
where we can define the complexes $\mC_0^* = \mC_{-1}^*:= L$ concentrated in dimension $0$ and hence 
$H^*(\mC_0^*) = H^0(\mC_0) = H^*(\mC_{-1}) = H^0(\mC_{-1}) = L.$

The statement of the theorem is trivially true for $n=1.$ 
Assume $n>1$ and suppose that the theorem holds for any integer $m,$ $m<n.$ 
Each non-trivial entry $E_1^{s,t}$ of the $E_1$ term of the spectral sequence for $\mC_n^*$ is 
isomorphic either to a $L$-module of the form $\q{h},$ for a suitable $h,$ or to the ring $L$ itself. 
The second case holds only for the entries $E^{0,n-1}_1$ and $E^{0,n}_1.$

The cyclotomic polynomials $\ph_h(q)$ are prime polynomials in the ring $L.$
As a consequence any map $d:\q{h} \rightarrow \q{k}$ induced by a differential $d_k$ of
of the spectral sequence can be non-zero only if $h = k$ and the map is an isomorphism.
In a similar way any map $d:\q{h} \rightarrow L/([n+1])$ can be non-zero only if 
$h \mid n+1$ and if $h \nmid n+1$ any map from $\q{h}$ to any quotient of  $L/([n+1])$ is trivial. This follows since 
$[n+1]$ is the product of the cyclotomic polynomials $\ph_h(q)$ for $h \mid n+1$ and
hence the $L$-module $L/([n+1])$ decomposes as a direct sum of modules $$L/([n+1]) = \bigoplus_{h|n+1}\q{h}.$$

Since the $L$-module $E^{n-1,0}_1 = L$ is generated by $01^{n-1}$ and $E^{n,0}_1 = L$
is generated by $1^n,$ from Example \ref{ex:rank_one} we can see that the differential 
$d_1: E^{n-1,0}_1 \rightarrow E^{n,0}_1$ is given by multiplication by $[n+1]$ and hence 
we have $E^{n-1,0}_2= 0$ and $E^{n,0}_2 = R/([n+1]).$ 

As a consequence if we fix a certain integer $h$ we can study the spectral
sequence considering only the 
terms isomorphic to $\q{h}$ and, if $h \mid n+1,$
the summand of $R/([n+1])$ isomorphic to $\q{h},$ while we can ignore all
the other summand in the spectral sequence.
We have three different cases:

\textit{i)}
$h \mid n$ \\
By induction we know that $E^{s,t}_1 = H^t(\mC^*_{n-s-1}) = \q{h}$ only in two cases:
\begin{itemize}
\item[\textit{i.a)}] $h \mid n-s-1$ and $ t = n-s-1-2 \frac{n}{h}+1;$
\item[\textit{i.b)}] $h \mid n-s$
and $t=n-s-1-2(\frac{n}{h} - 1 ).$ 
\end{itemize}
If we set $i := \frac{n}{h},$ in case \textit{i.a)} we have
\begin{eqnarray*}
\lambda = 1, \ldots, i-1 & E_1^{\lambda h-1,n-\lambda(h-2) -2i+1} &
\mbox{generated by } z_h(i-\lambda)01^{\lambda h-1}\\
\mbox{and in case in case \textit{i.b)} we have} & & \\
\lambda = 0, \ldots, i-1 & E_1^{\lambda h  ,n-\lambda(h-2) -2i+1} &
\mbox{generated by } v_h(i-\lambda)01^{\lambda h}.
\end{eqnarray*}
We note that 
$z_h(i)01^l = v_h(i)001^l - (-1)^{h(i-1)} w_h(i-1) 1^{h-1}01^l,$  hence we get
that the map
$d_1:E_1^{\lambda h-1,n-\lambda(h-2) -2i+1} \rightarrow E_1^{\lambda h,n-\lambda(h-2) -2i+1}$
is given by multiplication by
$[\lambda h + 1],$ so it is an isomorphism. It follows (see diagram below) that
the $L$-module $E_1^{0,n-21+1}$ is the only one that survives in the term 
$E_{\infty}$ and hence $E_1^{0,n-21+1}$ will give, as we will see next, 
the only contribution from $E_{\infty}$ to the cohomology group $H^{n-2i + 1}(\mC_n^*).$

\begin{center}
\begin{tabular}{|l}
\xymatrix @R=1pc @C=1pc {
\q{h} & & &            &       & & &                    &        & & &                   &\\
& & & \q{h}\ar[r]^\sim & \q{h} & & &                    &        & & &                   &\\
& & &                  &       & & & \cdots \ar[r]^\sim & \cdots & & &                   &  \\
& & &                  &       & & &                    &        & & & \q{h} \ar[r]^\sim & \q{h} \\
& & &                  &       & & &                    &        & & &                   & }\\
\hline
\end{tabular}
\end{center}
In order to consider the case \textit{ii)} we need the following lemma:
\begin{lem}[\cite{dss}] \label{l:dss} Let $1 = d_1 <  ...< d_n$ be the divisors of $n,$ in the ring $L$ we have the following 
equality of ideals:
$$
\left( \left[ \begin{array}{c}n\\d_1 \end{array} \right] ,...,\left[ \begin{array}{c}n\\d_k \end{array} \right] \right)=
(\varphi_{d_{k+1}}\cdots\varphi_n).
$$
\end{lem}

\textit{ii)}
$h \mid n+1$\\
Now we set $i := \frac{n+1}{h}.$ The two possible cases for $E^{s,t}_1 = \q{h}$ are 
the following ones:
\begin{eqnarray*}
\lambda = 1, \ldots, i-1 & E_1^{\lambda h-2,n-\lambda(h-2) -2i+2} &
\mbox{generated by } z_h(i-\lambda)01^{\lambda h-2}\\
\lambda = 1, \ldots, i-1 & E_1^{\lambda h-1,n-\lambda(h-2) -2i+2} &
\mbox{generated by } v_h(i-\lambda)01^{\lambda h-1}.
\end{eqnarray*}
The differential 
$d_1:E_1^{\lambda h-2,n-\lambda(h-2) -2i+2} \rightarrow E_1^{\lambda h-1,n-\lambda(h-2) -2i+2}$
is multiplication by the $q$-analog $[\lambda h]$ and hence it is the trivial map. 
The next differential that we need to consider is $$d_{h-1}:E_{h-1}^{\lambda h-1,n-\lambda(h-2) -2i+2}
\rightarrow E_{h-1}^{(\lambda+1) h-2,n-(\lambda+1)(h-2) -2i+2}.$$ 
The equality 
$v_h(i)01^l=z_h(i-1)01^{h-2}01^l+(-1)^{h(i-1)}w_h^{i-1}1^{h-1}01^l$ implies that
the map above corresponds to multiplication by
$[\lambda h+1]\ldots[\lambda h+h-1]$ and hence it is an isomorphism, since all the factors
are invertible in $\q{h}.$ Finally the map
$d_{h-1}:E_{h-1}^{n-h,h-1} \rightarrow E_{h-1}^{n,0}$ corresponds
to multiplication by
$\alpha_h = \left[ \begin{array}{c}n+1\\h \end{array} \right]$ and hence from Lemma \ref{l:dss} 
it is injective. Below we have a picture of the spectral sequence, with differentials $d_1$ and $d_{h-1}.$
We can see the there is only one nontrivial $\q{h}$-module that survives in
$E_{\infty},$ that is $E_{h-1}^{h-2,n-h+2-2(i-1)},$ that gives a contribution (actually the only one) 
to the cohomology group $H^{n-2(i-1)}(C_n).$
\begin{center}
\begin{tabular}{|l}
\xymatrix @R=1pc @C=1pc {
& & \q{h}\ar[r]^0 & \q{h}\ar[rrrd]^\sim& & &        &        &                       & &                &                           & & & \\
& &               &                    & & &\cdots \ar[r]^0 & \cdots \ar[rrrd]^\sim & &  &              &                           & & &\\
& &               &                    & & &        &        &                       & & \q{h} \ar[r]^0 & \q{h} \ar[rrrd]^{\alpha_h}& & &\\
& &               &                    & & &        &        &                       & &                &                           & & & R/I}\\
\hline
\end{tabular}
\end{center}
%

\textit{iii)}
$h \nmid n(n+1)$\\
Let $c, 1 < c < h$ be an integer such that $h\mid n+c,$ if we set $i := \frac{n+c}{h},$ 
we have again two cases for $E^{s,t}_1 = \q{h}$:
\begin{eqnarray*}
\lambda = 1, \ldots, i -1 & E_1^{\lambda h-c-1,n+c-\lambda(h-2) -2i+1}
& \mbox{generated by } z_h(i-\lambda)01^{\lambda h-c-1}\\
\lambda = 1, \ldots, i -1 & E_1^{\lambda h-c  ,n+c-\lambda(h-2) -2i+1}
& \mbox{generated by } v_h(i-\lambda)01^{\lambda h-c}.
\end{eqnarray*}
The differential 
$d_1:E_1^{\lambda h-c-1,n+c-\lambda(h-2) -2i+1}\rightarrow E_1^{\lambda h-c  ,n+c-\lambda(h-2) -2i+1}$
corresponds to multiplication by $[\lambda h -c +1]$ that is co-prime with $[h]$ and hence the map is 
an isomorphism. It follows that none of the modules survives in $E_2$ and 
hence the contribution to $E_{\infty}$ is trivial.

From Lemma \ref{l:dss} and from the previous observations in case \textit{ii)} we get that
$E_{\infty}^{n,0} =\q{n+1},$ generated by $1^n.$
From the description of the spectral sequence it follows that the cohomology
group $H^*(\mC_n^*)$ is the one described in the statement of the theorem.
In order to complete the proof we need to check that the generators are correct.
In case \textit{i)} the $L$-module $E_1^{0,n-21+1}$ is generated by $v_h(i)0$ that differs from $z_h(i)$
by a term of the form $A1$ and hence we can lift $v_h(i)0$ to $z_h(i).$
The case \textit{ii)} is analog: the $L$-module $E_{h-1}^{h-2,n-h+2-2(i-1)}$ is generated by
$z_h(i-1)01^{h-2}$ that differs from $v_h(i)$ by a term of the form $A1^{h-1}$ and hence
we can lift $z_h(i-1)01^{h-2}$ to $v_h(i).$
\end{proof}

\subsection{Artin group of affine type and non-linear Coxeter graphs: some remarks} \label{ss:affine}
We deal now with the case of an affine reflection group. 
Let $(W, S)$ be an affine reflection group with Coxeter graph $\Gamma$ and suppose $\mid \! S \! \mid =n+1.$ 
Let $\lambda$ be an abelian representation of the Artin group $A_\Gamma$ over a ring $R$ that is 
an unique factorization domain.
The generators of the Salvetti complex $(\mC^*, \delta)$ are in $1$ to $1$ correspondence
with the proper subsets of $S.$ It can be somewhat convenient to complete the complex
$\mC^*$ to an \emph{augmented} Salvetti Complex $\widehat{\mC}^*$ as follows:
$$
\widehat{\mC}^* := \mC^* \oplus R.e_S.
$$
We can define the coboundary $\widehat{\delta}$ on the complex 
$\widehat{\mC}^*$ setting 
$\widehat{\delta}(e_T) = \delta(e_T)$ if $\mid \! T \! \mid <n$ and
re-defining the coboundary on the 
top-dimensional generators of $\mC^*.$
We formally define a suitable \emph{quasi}-Poincar\'e polynomial for $W$ by:
$$
(\widehat{W}_S)_\lambda := \mathrm{lcm}\{(W_{S \setminus \{s\}})_\lambda \mid s \in S\}.
$$
and for every $s \in S$ we set the coboundary for $\widehat{\mC}^*$:
$$
\delta(e_{S \setminus \{ s\}}):= (-1)^{\sigma(s, S\setminus \{ s\})+1} 
\frac{(\widehat{W}_S)_\lambda}{(W_{S \setminus \{s\}})_\lambda}.
$$
 and it is straightforward to verify that $\widehat{\mC}^*$ is still a chain complex. Moreover, we have
the following relations between the cohomology of $\mC^*$ and $\widehat{\mC}^*$:
$$ H^i(\mC^*) = H^i (\widehat{\mC}^*) $$
for $i \neq n, n + 1$ and we have the short exact sequence
$$ 0 \to H^n (\widehat{\mC}^*) \to H^n (\mC^*) \to R \to 0.$$
An example of this construction can be found in the computation of the cohomology
of the affine Artin group of type $\widetilde{\mathbb{B}}_n$ in \cite{cms10}.

\subsection{A non-abelian case: three strands braid group and a geometric representation} \label{ss:nonab}
The third braid group $\Br_3$ and the special linear 
group $SL_2(\Z)$ have a classical geometric representation
given by symmetric power of the natural symplectic representation. The cohomology of this representation
is studied in detail in \cite{ccs12}. The aim of this section is to show how the Salvetti complex can be used 
for finite computations, even with non-abelian representation.

In general we can consider an orientable
surface $M_{g,n}$ of genus $g$ with $n$ connected components in its boundary and the isotopy
classes of Dehn twists around simple loops $c_1, \ldots, c_{2g}$ such that $\mid \! \! c_i \cap c_{i+1} \! \! \mid \; = 1$
and $\mid \! \! c_i \cap c_j \! \! \mid \; = 0$ if $j \neq i \pm 1.$ We give a representation of the braid group
in the symplectic group $Aut(H^*(M_{g,n};\Z); < >)$ of all automorphisms preserving the intersection form 
as follows: the $i$-th generator of the braid group
$\Br_{2g+1}$ maps to the Dehn twist with respect to the simple loop $c_i.$  
In the case $g=1,$ $n=1$ the symplectic group equals $SL_2(\Z).$ We extend this representation
to a representation $\lambda$ on the symmetric algebra $M=\Z[x,y].$ 
This representation splits into irreducible $SL_2(\Z)$-modules $M = \oplus_{n \geq 0}M_n$ 
according to the polynomial degree. In \cite{ccs12} the cohomology groups $H^*(\Br_3; M)$ and $H^*(SL_2(\Z); M)$
are computed. The main ingredients for the achievement of this result are 
the computation of the spectral sequence associated
to the central extension $$1 \to \Z \to \Br_3 \to SL_2(\Z) \to 1$$
(see \cite[Th. 10.5]{milnor71}), the amalgamated free product decomposition 
$$SL_2(\Z) = \Z_4 \ast_{\Z_2} \Z_6$$ (see \cite{mks}) and a generalization of 
a classical result of Dickson (see \cite{dickson, steinberg}) on the characterization 
of $SL_2(\F_p)$-invariants polynomials.

The methods described in this survey don't seem very useful to compute explicitly 
the cohomology group $H^*(\Br_3; M),$ 
but they can be used to get finite computations with the help of a computer.
In particular, for a fixed degree $n$ the computation of the cohomology group $H^*(\Br_3; M_n)$
is a very simple problem. 

Let $\sigma_1$ and $\sigma_2$ be the standard generators of the braid group $\Br_3.$ The action of the
representation $\lambda$ on degree-one polynomials is given by
 $$\sigma_1:\left\{ \begin{array}{l} x\to x-y\\ y \to y\end{array}\right. ,\  
\sigma_2:\left\{ \begin{array}{l} x\to x\\ y \to x+ y\end{array} \right. $$
and hence, with respect to the basis $\{x,y\}$ of $M_1,$ the representation is given by the matrices
$$
\sigma_1\stackrel{\lambda}{\mapsto} \begin{bmatrix} 1&0\\-1&1\end{bmatrix},\quad  
\sigma_2\stackrel{\lambda}{\mapsto} \begin{bmatrix} 1& 1\\ 0& 1\end{bmatrix}. $$ 
The action extends to the $n$-th symmetric algebra of the space $<x,y>,$ with basis
$\{ x^n, x^{n-1}y, \ldots, y^n\},$ by the matrices
$$ 
\sigma_1\stackrel{\lambda}{\mapsto} \begin{bmatrix} 
\binom{n}{0}& 0 & 0 & \cdots &  0\\
-\binom{n}{1} & \binom{n-1}{0} & 0 & \ddots & 0 \\
\binom{n}{2} & -\binom{n-1}{1} & \binom{n-2}{0} & \ddots & 0\\
\vdots  & \vdots & \ddots & \ddots & 0 \\
(\!-\!1\!)^n \binom{n}{n} & (\!-\!1\!)^{n\!-\!1} \binom{n-1}{n-1} & \cdots & -\binom{1}{1} & \binom{0}{0}
\end{bmatrix}, \quad  
\sigma_2\stackrel{\lambda}{\mapsto} \begin{bmatrix} 
\binom{0}{0} & \binom{1}{0} & \cdots & \binom{n -  1}{\;\;\,0} & \binom{n}{0}\\ 
0 & \binom{1}{1} & \ddots &  \binom{n-1}{1}& \binom{n}{1} \\
0 & 0 & \ddots & \ddots & \vdots\\ 
\vdots & \ddots & \ddots & \binom{n  -  1}{n  - 1} & \binom{n}{n  -  1} \\
0 & 0 & \cdots & 0 & \binom{n}{n}
\end{bmatrix}
$$ 
that is, $(\lambda(\sigma_1))_{ij}= (-1)^{i-j}\binom{n+1-j}{i-j}$ and $(\lambda(\sigma_2))_{ij}= \binom{j-1}{i-1},$ 
where $\binom{h}{k} = 0$ if $k<0.$
Hence we have to compute the cohomology for the complex $\mC^*$ given by:
\begin{center}
\begin{tabular}{l}
\xymatrix @R=1pc @C=1pc {
& *+<3pt>[F]{00} & \\
*+<3pt>[F]{10} \ar[ur]^{\sigma_1\sigma_2 - \sigma_2 + \Id} & & *+<3pt>[F]{01} \ar[ul]_{-\sigma_2\sigma_1 + \sigma_1 - \Id}\\
& *+<3pt>[F]{00} \ar[ul]^{\sigma_1-\Id} \ar[ur]_{\sigma_2 -\Id}& }\\
\end{tabular}
\end{center}
Similar computations for large $n$ can provide an evidence for general results like  in \cite{ccs12}. 

The reader familiar with computing local system (co)homology using resolutions will see that the cochain complex 
obtained here coincides with that obtained from the standard presentation of $\Br_3$ by these other methods.

The cochain complex above can be easily generalized to the case $g>1$, 
that is the computation of the co\-ho\-mo\-lo\-gy of the group $\Br_{2g+1}$
with coefficients in the representation on the ring of polynomials $\Z[x_1, y_1, \cdots, x_g, y_g].$

\providecommand{\bysame}{\leavevmode\hbox
to3em{\hrulefill}\thinspace}
\def\MR#1{MR#1}

\providecommand{\MRhref}[1]{%
  \href{http://www.ams.org/mathscinet-getitem?mr=#1}{#1}
} 
\providecommand{\href}[2]{\hyperlinks{#1}{#2}}

%
%
%

%
%
%
%
%
%
%
%
%
%


\bibliographystyle{amsalpha}
\bibliography{biblio}
\nocite{*}


\end{document}